\documentclass[letterpaper]{amsart}
\pdfoutput=1
\newtheorem{theorem}{Theorem}[section]

\newtheorem{proposition}[theorem]{Proposition}
\newtheorem{corollary}[theorem]{Corollary}

\theoremstyle{definition}

\theoremstyle{remark}

\usepackage{hyperref}
\usepackage{graphicx}
\usepackage{color}
\usepackage{amsmath}
\usepackage{amsfonts}
\usepackage{euler}
\usepackage{amssymb}
\usepackage{epsfig}
\usepackage{rotating}
\usepackage{graphics}
\newcommand{\be}{\begin{equation}}

\newcommand{\ee}{\end{equation}}

\newcommand{\dz}{\wedge}

\newcommand{\ba}{\begin{array}}

\newcommand{\ea}{\end{array}}

\newcommand{\beq}{\begin{eqnarray}}

\newcommand{\eeq}{\end{eqnarray}}

\newtheorem{lm}{lemma}

\newtheorem{thee}{theorem}

\newtheorem{proo}{proposition}

\newtheorem{co}{corollary}

\newtheorem{rem}{remark}

\newtheorem{deff}{definition}

\newcommand{\bd}{\begin{deff}}

\newcommand{\ed}{\end{deff}}

\newcommand{\bl}{\begin{lm}}

\newcommand{\el}{\end{lm}}

\newcommand{\bp}{\begin{proo}}

\newcommand{\ep}{\end{proo}}

\newcommand{\bt}{\begin{thee}}

\newcommand{\et}{\end{thee}}

\newcommand{\bc}{\begin{co}}

\newcommand{\ec}{\end{co}}

\newcommand{\brm}{\begin{rem}}

\newcommand{\erm}{\end{rem}}

\newcommand{\der}{{\rm d}}

\usepackage{t1enc}
\def\frak{\mathfrak}

\newcommand{\newc}{\newcommand}

\let\ccdot\cdot
\def\cdot{\hbox to 2.5pt{\hss$\ccdot$\hss}}

\newc{\aR}{\mbox{\boldmath{$ R$}}}
\newc{\aS}{\mbox{\boldmath{$ S$}}}
\newc{\aT}{\mbox{\boldmath{$ T$}}}
\newc{\aW}{\mbox{\boldmath{$ W$}}}

\newc{\aK}{\mbox{\boldmath{$ K$}}}
\newc{\aL}{\mbox{\boldmath{$ L$}}}

\usepackage{amssymb}
\usepackage{amscd}

\newcommand{\bma}{\begin{pmatrix}}
\newcommand{\ema}{\end{pmatrix}}

\newc{\obstrn}[2]{B^{#1}_{#2}}

\newcommand{\rpl}                         
{\mbox{$
\begin{picture}(12.7,8)(-.5,-1)
\put(0,0.2){$+$}
\put(4.2,2.8){\oval(8,8)[r]}
\end{picture}$}}

\newcommand{\lpl}                         
{\mbox{$
\begin{picture}(12.7,8)(-.5,-1)
\put(2,0.2){$+$}
\put(6.2,2.8){\oval(8,8)[l]}
\end{picture}$}}

\usepackage{ifthen}

\newc{\tensor}[1]{#1}
\newc{\Mvariable}[1]{\mbox{#1}}
\newc{\down}[1]{{}_{#1}}
\newc{\up}[1]{{}^{#1}}

\newc{\JulyStrut}{\rule{0mm}{6mm}}
\newc{\midtenPan}{\mbox{\sf S}}
\newc{\midten}{\mbox{\sf T}}
\newc{\midtenEi}{\mbox{\sf U}}
\newc{\ATen}{\mbox{\sf E}}
\newc{\BTen}{\mbox{\sf F}}
\newc{\CTen}{\mbox{\sf G}}

\def\sideremark#1{\ifvmode\leavevmode\fi\vadjust{\vbox to0pt{\vss
 \hbox to 0pt{\hskip\hsize\hskip1em
 \vbox{\hsize3cm\tiny\raggedright\pretolerance10000
 \noindent #1\hfill}\hss}\vbox to8pt{\vfil}\vss}}}%

\newcommand{\Span}{\mathrm{Span}}

\numberwithin{equation}{section}

\newcounter{romenumi}
\newcommand{\labelromenumi}{(\roman{romenumi})}

\newcommand{\bbT}{\mathbb{T}}

\begin{document}
\title[Distributions, $7^{th}$ order ODE and Pleba\'nski metric]{Symmetric $(2,3,5)$ distributions, an interesting ODE of $7^{th}$ order and Pleba\'nski metric}

\vskip 1.truecm
\author{Daniel An} \address{SUNY Maritime College 6 Pennyfield Avenue, Throggs Neck, New York 10465}
\email{dan@sunymaritime.edu}
\author{Pawe\l~ Nurowski} \address{Centrum Fizyki Teoretycznej,
Polska Akademia Nauk, Al. Lotnik\'ow 32/46, 02-668 Warszawa, Poland}
\email{nurowski@cft.edu.pl}
\thanks{This research was supported by the
  Polish National Science Centre under the grant DEC-2013/09/B/ST1/01799.}
\dedicatory{Dedicated to the memory of Sotirios Bonanos}
\date{\today}
\begin{abstract}
We show that the unique $7^{th}$ order ODE having 10 contact symmetries appears naturally in the theory of generic 2-distributions in dimension five.
\end{abstract}
\maketitle
\vspace{-1truecm}
\newcommand{\bbS}{\mathbb{S}}
\newcommand{\bbR}{\mathbb{R}}
\newcommand{\sog}{\mathbf{SO}}
\newcommand{\slg}{\mathbf{SL}}
\newcommand{\og}{\mathbf{O}}
\newcommand{\soa}{\frak{so}}
\newcommand{\sla}{\frak{sl}}
\newcommand{\sua}{\frak{su}}
\newcommand{\dr}{\mathrm{d}}
\newcommand{\sug}{\mathbf{SU}}
\newcommand{\gat}{\tilde{\gamma}}
\newcommand{\Gat}{\tilde{\Gamma}}
\newcommand{\thet}{\tilde{\theta}}
\newcommand{\Thet}{\tilde{T}}
\newcommand{\rt}{\tilde{r}}
\newcommand{\st}{\sqrt{3}}
\newcommand{\kat}{\tilde{\kappa}}
\newcommand{\kz}{{K^{{~}^{\hskip-3.1mm\circ}}}}
\newcommand{\bv}{{\bf v}}
\newcommand{\di}{{\rm div}}
\newcommand{\curl}{{\rm curl}}
\newcommand{\cs}{(M,{\rm T}^{1,0})}
\newcommand{\tn}{{\mathcal N}}
\section{Introduction}
Recently Dunajski and Sokolov have found \cite{DS} a general solution to an interesting $7^{th}$ order ODE:
$$10{y^{(3)}}^3y^{(7)}-70{y^{(3)}}^2y^{(4)}y^{(6)}-49{y^{(3)}}^2{y^{(5)}}^2+280y^{(3)}{y^{(4)}}^2y^{(5)}-175{y^{(4)}}^4=0.$$
This equation can be characterized as a unique (modulo contact transformations of variables) $7^{th}$ order ODE which has a 10-dimensional group of local contact symmetries \cite{noth,olver,sok}. As mentioned by Dunajski and Sokolov this equation was known already in 1904 by Noth. Since we cannot find any earlier reference to this equation we will call this \emph{Noth's equation} in the following.

In this short note we show that Noth's equation also turns out to be a natural geometric condition for a certain class of generic 2-distributions in dimension five.  

We say that a 2-distribution ${\mathcal D}=\Span(X_4,X_5)$, 
where $X_4$ and $X_5$ are two vector fields on a 5-dimensional manifold $M$, is \emph{generic}, or $(2,3,5)$ on $M$, if the system of five vector fields  
$$(~X_1,X_2,X_3,X_4,X_5~)~=~(~[X_5,[X_4,X_5]],~[X_4,[X_4,X_5]],~ [X_4,X_5],~X_4, ~X_5~)$$
forms a frame on $M$. Locally, a generic 2-distribution $\mathcal D$ on $M$ can be defined as the annihilator of three 1-forms $(\omega_1,\omega_2,\omega_3)$ on $M$, defined in terms of a single real function $f=f(x,y,p,q,z)$, such that $f_{qq}\neq 0$, via:
\be
\omega_1=\der y-p\der x,\quad\quad\omega_2=\der p-q\der x,\quad\quad\omega_3=\der z-f(x,y,p,q,z)\der x.\label{fog}\ee
Here $(x,y,p,q,z)$ is a local coordinate system in $M$. 

The local geometry of $(2,3,5)$ distributions is nontrivial: there exist generic distributions ${\mathcal D}_1$ and ${\mathcal D}_2$ on $M$ which do \emph{not} admit a local diffeomorphism $\varphi:M\to M$ such that $\varphi_*{\mathcal D}_1={\mathcal D}_2$. For example, distributions corresponding to a function $f=q^2$ and $f=q^3$ in (\ref{fog}) do not admit such a diffeomorphism. In such case we say that they are \emph{locally nonequivalent}. The full set of differential invariants of $(2,3,5)$ distributions considered modulo local diffeomorphism was given by Cartan in \cite{cartan}. For each $(2,3,5)$ distribution he associated a Cartan connection, with values in the split real form of the exceptional Lie algebra $\mathfrak{g}_2$, whose curvature provided the invariants. These invariants can be also understood in terms of a certain conformal class of metrics \cite{nurdif}, defined on $M$ by $\mathcal D$. This conformal class is defined as follows:

Let $\mathcal D$ be defined as the annihilator of 1-forms $(\omega_1,\omega_2,\omega_3)$, as e.g. in (\ref{fog}). We supplement them by the 1-forms $\omega_4$ and $\omega_5$, in such a way that $\omega_1\dz\omega_2\dz\omega_3\dz\omega_4\dz\omega_5\neq 0$. In case of (\ref{fog}) we take $\omega_4=\der q$ and $\omega_5=\der x$. Consider the forms $(\theta^1,\theta^2,\theta^3,\theta^4,\theta^5)$ defined on $M$ via:
\be \label{matrix}
\bma\theta^1\\\theta^2\\\theta^3\\\theta^4\\\theta^5\ema=\bma b_{11}&b_{12}&b_{13}&0&0\\
b_{21}&b_{22}&b_{23}&0&0\\
b_{31}&b_{32}&b_{33}&0&0\\
b_{41}&b_{42}&b_{43}&b_{44}&b_{45}\\
b_{51}&b_{52}&b_{53}&b_{54}&b_{55}
\ema\bma\omega^1\\\omega^2\\\omega^3\\\omega^4\\\omega^5\ema,
\ee
with some functions $b_{ij}$, $i,j=1,2,\dots,5$, on $M$ such that $\theta^1\dz\theta^2\dz\theta^3\dz\theta^4\dz\theta^5\neq 0$. 
It follows that for a $(2,3,5)$ distribution $\mathcal D$ one can always find functions $b_{ij}$ and 1-forms $\Omega_\mu$, $\mu=1,2,\dots,7$, on $M$ such that  
\be \label{system}
\begin{aligned}
&\der\theta^1=\theta^1\dz(2\Omega_1+\Omega_4)+\theta^2\dz\Omega_2+\theta^3\dz\theta^4\\
&\der\theta^2=\theta^1\dz\Omega_3+\theta^2\dz(\Omega_1+2\Omega_4)+\theta^3\dz\theta^5\\
&\der\theta^3=\theta^1\dz\Omega_5+\theta^2\dz\Omega_6+\theta^3\dz(\Omega_1+\Omega_4)+\theta^4\dz\theta^5\\
&\der\theta^4=\theta^1\dz\Omega_7+\tfrac43\theta^3\dz\Omega_6+\theta^4\dz\Omega_1+\theta^5\dz\Omega_2\\
&\der\theta^5=\theta^2\dz\Omega_7-\tfrac43\theta^3\dz\Omega_5+\theta^4\dz\Omega_3+\theta^5\dz\Omega_4.
\end{aligned}
\ee
And now, it turns out that the $(3,2)$-signature conformal class $[g_{\mathcal D}]$ represented on $M$ by the metric 
\be
g_{\mathcal D}=g_{ij}\theta^i\otimes\theta^j=\theta^1\otimes\theta^5+\theta^5\otimes\theta^1-\theta^2\otimes\theta^4-\theta^4\otimes\theta^2+\tfrac43\theta^3\otimes\theta^3\label{met}\ee
is well defined, and that its Weyl tensor can be used to get all the basic differential invariants of the distribution $\mathcal D$. The simplest of these invariants, the so called \emph{Cartan's quartic} $C(\zeta)$ of $\mathcal D$, can be obtained in terms of the Weyl tensor of the conformal class $[g_{\mathcal D}]$ as follows \cite{twist,Gra}:

Calculate the Weyl tensor $W=W_{ijkl}\theta^i\otimes\theta^j\otimes\theta^k\otimes\theta^l$ for the metric $g_{\mathcal D}$ in the coframe $(\theta^1,\theta^2,\theta^3,\theta^4,\theta^5)$. (Use the metric $g_{\mathcal D}$ to lower the index $i$ from the natural placement $W^i_{~jkl}$ to $W_{ijkl}$, $W_{ijkl}=g_{ip}W^p_{~jkl}$). Then Cartan's quartic for $\mathcal D$ is 
$$\begin{aligned}
C(\zeta)~:=~&A_1+4A_2\zeta+6A_3\zeta^2+4A_4\zeta^3+A_5\zeta^4~
\end{aligned}$$
with the functions $A_I$, $I=1,2,\dots,5$, given by:
$$
A_1=W_{4114},\quad  A_2=W_{4124},\quad A_3=W_{4125},\quad A_4=W_{4225},\quad   A_5=W_{5225}.
$$

The simplest equivalence class of $(2,3,5)$ distributions corresponds to the vanishing of Cartan's quartic, $C(\zeta)\equiv 0$, or equivalently, $A_I\equiv 0$ for all $I=1,2,\dots,5$. Modulo local diffeomorphisms there is only one such distribution $\mathcal D$. It may be represented by (\ref{fog}) with $f=q^2$. This distribution has maximal group of local symmetries. This group is isomorphic to the split real form of the exceptional group $G_2$ \cite{cartan}. 

This provokes a problem: find all functions $f=f(x,y,p,q,z)$, which via (\ref{fog}), define a generic distribution $\mathcal D$, which is locally diffeomorphically equivalent to the most symmetric one, the one with $f=q^2$.

The general solution to this problem requires rather elaborate calculations, and it follows that the PDEs required for $f$ to correspond to vanishing $A_I$s are quite ugly. However in the restricted case when the function $f$ depends only on a single variable $q$ the solution is quite nice (see \cite{nurdif}, equation (57)). For completeness we recall this solution in the next section.

\section{Cartan quartic for the distribution with $f=f(q)$}   
If the distribution is given as the annihilator of 
\be\begin{aligned}
\omega_1&=\der y-p\der x\\
\omega_2&=\der p-q\der x\\
\omega_3&=\der z-f(q)\der x,\end{aligned}\label{ff}\ee
the conformal class $[g_{\mathcal D}]$ can be represented by (\ref{met}), with the forms $(\theta^1,\theta^2,\theta^3,\theta^4,\theta^5)$ given by:
$$\begin{aligned}
\theta^1&=\omega_1-\frac{1}{f''}(f'\omega_2-\omega_3)\\
\theta^2&=\frac{1}{f''}(f'\omega_2-\omega_3)\\
\theta^3&=\frac{4{f''}^2-f'f^{(3)}}{4{f''}^2}\omega_2+\frac{f^{(3)}}{4{f''}^2}\omega_3\\
\theta^4&=\frac{(7{f^{(3)}}^2-4f''f^{(4)})}{40{f''}^3}(f'\omega_2-\omega_3)+\omega_4-\omega_5\\
\theta^5&=-\omega_4,
\end{aligned}
$$
where
$$\omega_4=\der q,\quad\quad\omega_5=\der x.$$
With this choice of $\theta^i$s the Cartan quartic is
$$C(\zeta)=\frac{a_5}{100{f''}^4}\zeta^4$$
with
$$a_5=10f^{(6)}{f''}^3-80{f''}^2f^{(3)}f^{(5)}-51{f''}^2{f^{(4)}}^2+336f''{f^{(3)}}^2f^{(4)}-224{f^{(3)}}^4.$$
We see, in particular, that the only nonvanishing component of Cartan's quartic is $A_5$, and that the quartic has a \emph{quadruple} root, which makes it of type $IV$, in the terminology of \cite{franco}.
  
We have the following corollary.
\begin{corollary}
Necessary and sufficient conditions for the distribution 
$${\mathcal D}=\Span(\partial_q,\partial_x+p\partial_y+q\partial_p+f(q)\partial_z)$$ 
to have split real form of the exceptional Lie group $G_2$ as a group of its local symmetries are:
\be 
f''\neq 0\label{ope}\ee
and
\be
10f^{(6)}{f''}^3-80{f''}^2f^{(3)}f^{(5)}-51{f''}^2{f^{(4)}}^2+336f''{f^{(3)}}^2f^{(4)}-224{f^{(3)}}^4=0.\label{a5}\ee
\end{corollary}
Thus apart from the genericity condition (\ref{ope}) the function $f$ must satisfy quite a complicated $6^{th}$ order ODE (\ref{a5}).

Strangely enough, this ODE is closely related to the Noth equation, studied by Dunajski and Sokolov, and mentioned in the Introduction.
We have the following proposition.
\begin{proposition}\label{pt}
Suppose that a real, sufficiently many times differentiable, function $f=f(q)$ satisfies equation (\ref{a5}). Let $\Theta=\Theta(x_5)$ be another real, sufficiently many times differentiable, function of a real variable $x_5$, whose second and third derivatives with respect to $x_5$ are related to $f$ via an equation:
\be
f(-\Theta^{(3)})+x_5 \Theta^{(3)}-\Theta''=0.\label{fthe}\ee Assume in addition that $\Theta^{(4)}\neq 0$. Then the function $\Theta=\Theta(x_5)$ satisfies the following $8^{th}$ order ODE:
$$10{\Theta^{(4)}}^3\Theta^{(8)}-70{\Theta^{(4)}}^2\Theta^{(5)}\Theta^{(7)}-49{\Theta^{(4)}}^2{\Theta^{(6)}}^2+280\Theta^{(4)}{\Theta^{(5)}}^2\Theta^{(6)}-175{\Theta^{(5)}}^4=0.$$
\end{proposition}
\begin{proof}
The proof consists in a successive differentiation of the equation $f(-\Theta^{(3)})=-x_5 \Theta^{(3)}+\Theta''$ using the chain rule. We have: $-\Theta^{(4)}f'=-\Theta^{(3))}-x_5\Theta^{(4)}+\Theta^{(3)}$,  i.e. $f'=x_5$. Then, in the same way:
$$f^{(p)}=-\frac{1}{\Theta^{(4)}}\frac{\der}{\der x_5}f^{(p-1)}\quad\quad {\rm for}\quad p=2,3,\dots,$$
i.e. $f''=-\frac{1}{\Theta^{(4)}}$, $f^{(3)}=-\frac{\Theta^{(5)}}{{\Theta^{(4)}}^3}$, $f^{(4)}=\frac{\Theta^{(6)}}{{\Theta^{(4)}}^4}-3\frac{{\Theta^{(5)}}^2}{{\Theta^{(4)}}^5}$, etc. Inserting these derivatives of $f$ into the definition of $a_5$ we get 
$$
\begin{aligned}
&a_5=\\
&-\frac{10{\Theta^{(4)}}^3\Theta^{(8)}-70{\Theta^{(4)}}^2\Theta^{(5)}\Theta^{(7)}-49{\Theta^{(4)}}^2{\Theta^{(6)}}^2+280\Theta^{(4)}{\Theta^{(5)}}^2\Theta^{(6)}-175{\Theta^{(5)}}^4}{{\Theta^{(4)}}^{12}}.
\end{aligned}
$$
Thus if the equation $a_5=0$ for $f$ is satisfied, i.e. if (\ref{a5}) holds, then the function $\Theta=\Theta(x_5)$ satisfies the $8^{th}$ order ODE from the proposition, as claimed.
\end{proof}

Magically, the equation (\ref{a5}), when transformed via (\ref{fthe}) into the $8^{th}$ order ODE from Proposition \ref{pt} and then reduced by one order via $y=\Theta'$, becomes the $7^{th}$ order ODE of Noth. The magic is in a peculiar form of the transformation (\ref{fthe}) relating $f$ and $\Theta$. The geometric reason for this transformation is explained in the next section.
\section{Distribution with $f=f(q)$ as a twistor distribution}
A particular class of $(2,3,5)$ distributions is associated with 4-dimensional split signature metrics. This is carefully explained in \cite{twist}, see Section 2, for every split signature metric. Here we concentrate on a special case, when the metric is given in terms of a one real function of four variables, called \emph{Pleba\'nski second heavenly function}. 

Let $\Theta=\Theta(x,y,z,w)$ be a real, sufficiently smooth, function of four real variables $(x,y,z,w)$ which satisfies the differential equation called \emph{second heavenly equation of Pleba\'nski}:
\be
\Theta_{wx}+\Theta_{zy}+\Theta_{xx} \Theta_{yy} - \Theta_{xy}^2=0.\label{secpleb}
\ee
Such a function enables us to define a 4-metric $g$, on a manifold $\mathcal U$ parametrized by $(x,y,z,w)$, via:
\be
g=\der w\der x+\der z\der y-\Theta_{xx}\der z^2-\Theta_{yy}\der w^2+2\Theta_{xy}\der w \der z.\label{ple}\ee
This is Pleba\'nski's second heavenly metric. It can be written in the form
$$g=\tau^1\otimes\tau^2+\tau^2\otimes\tau^1+\tau^3\otimes\tau^4+\tau^4\otimes\tau^3,$$
where
$$\begin{aligned}
\tau^1&=\der x-\Theta_{yy}\der w+\Theta_{xy}\der z\\
\tau^2&=\der w\\
\tau^3&=\der y-\Theta_{xx}\der z+\Theta_{xy}\der w\\
\tau^4&=\der z.
\end{aligned}
$$ 
Since $g$ has split signature on $\mathcal U$, there is a natural circle bundle $\bbS^1\to \bbT({\mathcal U})\to \mathcal U$ over $\mathcal U$ \cite{twist}. In this bundle, which we call as the \emph{circle twistor bundle} for $({\mathcal U},g)$, every point in the fiber over $x\in \mathcal U$ is a certain \emph{real} totally null selfdual 2-plane at $x$. There is an entire circle of such planes at $x$. The bundle $\bbT({\mathcal U})\stackrel{\pi}{\to}\mathcal U$ is naturally equipped with a 2-dimensional distribution $\mathcal D$. Its plane ${\mathcal D}_p$ at a point $p\in\bbT({\mathcal U})$, which as we know can be identified with a certain real totally null 2-plane $N(p)$ at $\pi(p)$, is the tautological \emph{horizontal} lift of $N(p)$ from $\pi(p)$ to $p$. \emph{Horizontality} in $\bbT({\mathcal U})$ is induced by the Levi-Civita connection of $g$ from $\mathcal U$. (See \cite{twist}, Section 2, for details). In case of the Pleba\'nski metric (\ref{ple}), given in terms of the heavenly function $\Theta=\Theta(x,y,z,w)$, the circle twistor bundle can be locally parametrized by $(x,y,z,w,\xi)$ and the twistor distribution can be defined as the annihilator of the 1-forms 
$$
\begin{aligned}
\tilde{\omega}_1&=\der \xi-\Big((\partial_x+\xi\partial_y)^3\Theta\Big)\der z\\
\tilde{\omega}_2&=\der w+\xi\der z\\
\tilde{\omega}_3&=\der y-\xi\der x-\Big((\partial_x+\xi\partial_y)^2\Theta\Big)\der z.
\end{aligned}
$$
A little tweak (see \cite{franco}, Thm 3.3.5), which we have learned from Ian Anderson \cite{ia}, and which he attributes to Goursat \cite{gou} (see also \cite{bryant}, p.7), consists in introducing new coordinates $(x_1,x_2,x_3,x_4,x_5)$ on $\bbT({\mathcal U})$:
\be \label{coord}
x_1=z,\quad\quad x_2=w,\quad\quad x_3=-\xi,\quad\quad x_4=y-\xi x,\quad\quad x_5=x,\ee
and enables us to conclude that the twistor distribution for the Pleba\'nski metric (\ref{ple}) can equivalently be defined by the annihilator of the forms 
\be \begin{aligned}
\omega_1&=\der x_2-x_3\der x_1\\
\omega_2&=\der x_3+\Theta_{555}\der x_1\\ 
\omega_3&=\der x_4-(\Theta_{55}-x_5\Theta_{555})\der x_1.
\end{aligned}\label{ft}\ee
Here $\Theta_{55}=\frac{\partial^2\Theta}{\partial x_5^2}$, $\Theta_{555}=\frac{\partial^3\Theta}{\partial x_5^3}$, and because $\Theta$ is originally function of only \emph{four} variables $(x,y,z,w)$, we have $\Theta_3+x_5\Theta_4=0$. 

Now we can demystify transformation (\ref{fthe}): Consider the case when the function $\Theta$ is a function of $x$ only, $\Theta=\Theta(x_5)$. Note that in this case the second heavenly equation (\ref{secpleb}) is automatically satisfied. Then, comparing the formulae (\ref{ff}) and (\ref{ft}), we see that the relation between the function $f$ in (\ref{ff}) and the function $\Theta$ in (\ref{ft}) is 
$$q=-\Theta_{555},\quad\quad f(q)=\Theta_{55}-x_5\Theta_{555}.$$
This inevitably leads to 
$$f(-\Theta_{555})=\Theta_{55}-x_5\Theta_{555},$$
which is the relation (\ref{fthe}).

For an explicit derivation of Noth's ODE in terms of the Pleba\'nski second heavenly metric, we find explicit formulae for the conformal class $[g_{\mathcal D}]$ associated with the distribution ${\mathcal D}$ defined by (\ref{ft}) with $\Theta=\Theta(x_5)$. Since for this the function $\Theta$ is a function of \emph{one} variable only, we will denote the derivatives w.r.t. $x_5$ by primes, double primes, etc. 
First we extend the forms (\ref{ft}) by 
\be \label{ft2} \omega_4=\der x_1,\quad\quad\omega_5=\der x_5 \ee
to a coframe on $\bbT({\mathcal U})$, and then find a suitable representatives of the forms $(\theta^1,\theta^2,\theta^3,$ $\theta^4,\theta^5)$ defining, via (\ref{met}), the conformal class $[g_{\mathcal D}]$. These forms can be taken to be:
$$
\begin{aligned}
\theta^1&=\omega_1-\Theta^{(4)}(x_5\omega_2-\omega_3)\\
\theta^2&=\Theta^{(4)}(x_5\omega_2-\omega_3)\\
\theta^3&=-\frac{4\Theta^{(4)}+x_5\Theta^{(5)}}{4\Theta^{(4)}}\omega_2+\frac{\Theta^{(5)}}{4\Theta^{(4)}}\omega_3\\
\theta^4&=-\frac{5{\Theta^{(5)}}^2-4\Theta^{(4)}\Theta^{(6)}}{40{\Theta^{(4)}}^3}(x_5\omega_2-\omega_3)+\omega_4-\Theta^{(4)}\omega_5\\
\theta^5&=\Theta^{(4)}\omega_5,
\end{aligned}
$$
with
$$\begin{aligned}
\omega_1&=\der x_2-x_3\der x_1\\
\omega_2&=\der x_3+\Theta^{(3)}\der x_1\\
\omega_3&=\der x_4-(\Theta''-x_5\Theta^{(3)})\der x_1\\
\omega_4&=\der x_1\\
\omega_5&=\der x_5.
\end{aligned}
$$
It is straightforward now to calculate Cartan's quartic for $g_{\mathcal D}$ with these forms $\theta^i$. It reads:
$$C(\zeta)=-\frac{\alpha_5\zeta^4}{100\Theta^{(4)}},$$
where $\alpha_5$ is given by
$$\alpha_5=10{\Theta^{(4)}}^3\Theta^{(8)}-70{\Theta^{(4)}}^2\Theta^{(5)}\Theta^{(7)}-49{\Theta^{(4)}}^2{\Theta^{(6)}}^2+280\Theta^{(4)}{\Theta^{(5)}}^2\Theta^{(6)}-175{\Theta^{(5)}}^4.$$ 
Thus under the condition $\Theta^{(4)}\neq 0$, Cartan's quartic identically vanishes if and only if $y=\Theta'$ satisfies Noth's ODE. 

We have just proved the following theorem.
\begin{theorem}\label{theo31}
The twistor distribution $\mathcal D$ on the circle twistor bundle $\bbS^1\to\bbT(M)\to M$ of the Pleba\'nski second heavenly manifold $(M,g)$ with the metric 
$$g=\der w\der x+\der z\der y-\Theta''\der z^2$$
and the second heavenly function $\Theta=\Theta(x)$ such that $\Theta^{(4)}\neq 0$,  has the split real form of the exceptional group $G_2$ as a group of its local symmetries if and only if the heavenly function $\Theta$ satisfies the ODE:
$$10{\Theta^{(4)}}^3\Theta^{(8)}-70{\Theta^{(4)}}^2\Theta^{(5)}\Theta^{(7)}-49{\Theta^{(4)}}^2{\Theta^{(6)}}^2+280\Theta^{(4)}{\Theta^{(5)}}^2\Theta^{(6)}-175{\Theta^{(5)}}^4=0.$$
\end{theorem}

\section{$G_2$ flatness for $\mathcal D$ of Pleba\'nski metrics implies Noth's equation}

The aim of this section\footnote{The first version of this paper appeared on arXiv.org in 2013 without this section. Since then we have expanded our result to general Plebanski metrics. Also, the results of our arXiv.org version of the paper were used by M. Randal in \cite{randall1,randall2}}, is to argue that $G_2$ flatness of the twistor distribution associated with Pleba\'nski second heavenly manifolds $(M,g)$ implies the Noth's equation for the heavenly function $\Theta$. Since the Cartan quartic for the general case looks horrible and is impossible to be displayed here, we present the details of the calculation only in the restricted case when the heavenly function $\Theta$  is independent of $w$ and $z$. We will comment on the general case at the end of this section.

If we assume that $\Theta$ is independent of $w$ and $z$, we have $\Theta_w=0$ and $\Theta_z=0$, which degenerates the second heavenly equation into the form:
\be
\Theta_{xx} \Theta_{yy} - \Theta{xy}^2=0,\label{ma1}
\ee
which is the homogenous Monge-Ampere equation. This is a special case which is convenient to study using the same change of coordinates as before in (\ref{coord}), because in them the differential equation retains its form:
\be
\Theta_{44} \Theta_{55} - \Theta_{45}^2=0.\label{ma}
\ee
Also, the conditions $\Theta_w=\Theta_z=\Theta_\xi=0$ in the original coordinates imply, and are equivalent to, that the function $\Theta$ in the new coordinates satisfies:
$$\Theta_1=\Theta_2=\Theta_3+x_5\Theta_4=0.$$
It is convenient to rewrite the Monge-Ampere equation (\ref{ma}) into equivalent form,  
\be \begin{aligned}
\Theta_{45}&=H \Theta_{55} \\
\Theta_{44}&=H^2 \Theta_{55}, 
\end{aligned}
\ee
which is a system of differential equations for $\Theta$ and a certain differentiable function $H$. Then applying these relations, we can compute the exterior derivatives of $\Theta$, $H$ and their partials as follows: 
$$ \begin{aligned}
&d\Theta &=& -\Theta_4 x_5 dx_3 + \Theta_4 dx_4 + \Theta_5 dx_5\\
&d\Theta_5 &=& (-\Theta_4 - H \Theta_{55} x_5)dx_3 + H \Theta_{55} dx_4 + \Theta_{55} dx_5\\
&d\Theta_4 &=& -H^2 \Theta_{55} x_5 dx_3 + H^2 \Theta_{55} dx_4 + H \Theta_{55} dx_5\\
&d\Theta_{55} &=& -\frac{2 H^2 \Theta_{55} + H H_5 \Theta_{55} x_5 + H^2 \Theta_{555} x_5}{H} dx_3 + \frac{H H_5 \Theta_{55} + H^2 \Theta_{555}}{H} dx_4 \\
  & & &+ \Theta_{555} dx_5 \\
&d\Theta_{555} &=& (-3 H_5 \Theta_{55} - 3 H \Theta_{555} - H_{55} \Theta_{55} x_5 - 2 H_5 \Theta_{555} x_5 -H \Theta_{5555} x_5)dx_3 \\ 
      & & &+ (H_{55} \Theta_{55} + 2 H_5 \Theta_{555} + H \Theta_{5555})dx_4 +  \Theta_{5555} dx_5\\
&d\Theta_{5555} &=& (-4 H_{55} \Theta_{55} - 8 H_5 \Theta_{555} - 4 H \Theta_{5555} - H_{555} \Theta_{55} x_5 \\
      & & &-3 H_{55} \Theta_{555} x_5 - 3 H_5 \Theta_{5555} x_5 - H \Theta_{55555} x_5 )dx_3 \\
            &  &  &+ (H_{555} \Theta_{55} + 3 H_{55} \Theta_{555} + 3 H_5 \Theta_{5555} + H \Theta_{55555})dx_4 +  \Theta_{55555} dx_5\end{aligned}$$
$$\begin{aligned}
&d\Theta_{55555} &=& (-5 H_{555} \Theta_{55} - 15 H_{55} \Theta_{555} - 15 H_5 \Theta_{5555} - 5 H \Theta_{55555} - H_{5555} \Theta_{55} x_5 \\ 
& & &- 4 H_{555} \Theta_{555} x_5 - 6 H_{55} \Theta_{5555} x_5 - 4 H_5 \Theta_{55555} x_5 - H \Theta_{555555} x_5)dx_3 \\
       & & &+ (H_{5555} \Theta_{55} + 4 H_{555} \Theta_{555} + 6 H_{55} \Theta_{5555} + 4 H_5 \Theta_{55555} + H \Theta_{555555})dx_4 \\
       & & & + \Theta_{555555} dx_5\\
&d\Theta_{555555}&=&(-6 H_{5555} \Theta_{55} - 24 H_{555} \Theta_{555} - 36 H_{55} \Theta_{5555} - 24 H_5 \Theta_{55555} - 6 H \Theta_{555555} \\
       & & &- H_{55555} \Theta_{55} x_5 - 5 H_{5555} \Theta_{555} x_5 - 10 H_{555} \Theta_{5555} x_5 - 10 H_{55} \Theta_{55555} x_5 \\
          & & &- 5 H_5 \Theta_{555555} x_5 - H \Theta_{5555555} x_5)dx_3 \\
          & & &+ (H_{55555} \Theta_{55} + 5 H_{5555} \Theta_{555} + 10 H_{555} \Theta_{5555} + 10 H_{55} \Theta_{55555} \\
          & & & + 5 H_5 \Theta_{555555} + H \Theta_{5555555})dx_4 + \Theta_{5555555} dx_5\\
          \end{aligned} $$
$$ \begin{aligned}
&dH&=&H (H - H_5 x_5)dx_3 + H H_5 dx_4 + H_5 dx_5\\
&dH_5&=&(H H_5 - H_5^2 x_5 - H H_{55} x_5)dx_3 + (H_5^2 + H H_{55})dx_4 +  H_{55} dx_5\\
&dH_{55}&=&-(3 H_5 H_{55} + H H_{555})x_5 dx_3 + (3 H_5 H_{55} + H H_{555})dx_4 +  H_{555} dx_5\\
&dH_{555}&=&(-3 H_5 H_{55} - H H_{555} - 3 H_{55}^2 x_5 - 4 H_5 H_{555} x_5 -     H H_{5555} x_5)dx_3  \\
   & &  &+ (3 H_{55}^2 + 4 H_5 H_{555} + H H_{5555})dx_4 +  H_{5555} dx_5\\
&dH_{5555}&=&(-6 H_{55}^2 - 8 H_5 H_{555} - 2 H H_{5555} - 10 H_{55} H_{555} x_5 -5 H_5 H_{5555} x_5 \\
      & & &- H H_{55555} x_5)dx_3 \\
      & &  &+ (10 H_{55} H_{555} + 5 H_5 H_{5555} + H H_{55555})dx_4 + H_{55555} dx_5.
\end{aligned}
$$

  Now, we can calculate the components $A_i$, $i=1,2,...,5$ of the Cartan quartic for the twistor distribution associated with this Pleba\'nski function $\Theta$. We take the coframe given by (\ref{ft}) and (\ref{ft2}), and find a new coframe  (\ref{matrix}) such that (\ref{system}) is satisfied. It requires long steps of calculations to find such a coframe. The coframe we eventually found is shown below:
$$
\bma\theta^1\\\theta^2\\\theta^3\\\theta^4\\\theta^5\ema=
\bma 
\Theta_{5555} & 0& 0& 0& 0\\ 
0& -x_5& 1& 0& 0\\ 
0& 1& 0& 0&  0\\ 
b_{41} &b_{42} & \frac{-5 \Theta_{55555}^2 + 3 \Theta_{5555} \Theta_{555555} }{30 \Theta_{5555}^2} & 0& -\Theta_{5555}\\ 
b_{51}& b_{52}& 0& 1& 2 H \Theta_{55}
\ema
\bma\omega^1\\\omega^2\\\omega^3\\\omega^4\\\omega^5\ema,
$$
where 
     $$ \begin{aligned}
b_{41}&=-\frac{1}{15 \Theta_{5555}^2} (3 H_{5555} \Theta_{55}^2 \Theta_{5555} + 27 H_{555} \Theta_{55} \Theta_{555} \Theta_{5555}  +     45 H_{55} \Theta_{555}^2 \Theta_{5555} \\ 
     &+ 18 H_{55} \Theta_{55} \Theta_{5555}^2 + 45 H_{5} \Theta_{555} \Theta_{5555}^2 - 5 H_{555} \Theta_{55}^2 \Theta_{55555} \\ 
     &- 35 H_{55} \Theta_{55} \Theta_{555} \Theta_{55555}  - 40 H_{5} \Theta_{555}^2 \Theta_{55555} + 7 H_{5} \Theta_{55} \Theta_{5555} \Theta_{55555} \\ 
     & +     5 H \Theta_{555} \Theta_{5555} \Theta_{55555} + 5 H \Theta_{55} \Theta_{55555}^2 -     3 H \Theta_{55} \Theta_{5555} \Theta_{555555})\\
b_{42}&=\frac{10 \Theta_{5555} \Theta_{55555} + 5 \Theta_{55555}^2 x_5 - 3 \Theta_{5555} \Theta_{555555} x_5}{30 \Theta_{5555}^2}\end{aligned}$$
$$\begin{aligned}
b_{51}&=-\frac{1}{30 \Theta_{5555}^3} (-5 H_{555}^2 \Theta_{55}^4 - 70 H_{55} H_{555} \Theta_{55}^3 \Theta_{555}  - 245 H_{55}^2 \Theta_{55}^2 \Theta_{555}^2 \\
     &- 80 H_{5} H_{555} \Theta_{55}^2 \Theta_{555}^2  - 560 H_{5} H_{55} \Theta_{55} \Theta_{555}^3 - 320 H_{5}^2 \Theta_{555}^4 + 30 H_{55}^2 \Theta_{55}^3 \Theta_{5555} \\ 
     & +     20 H_{5} H_{555} \Theta_{55}^3 \Theta_{5555} - 6 H H_{5555} \Theta_{55}^3 \Theta_{5555} + 260 H_{5} H_{55} \Theta_{55}^2 \Theta_{555} \Theta_{5555} \\
     & - 44 H H_{555} \Theta_{55}^2 \Theta_{555} \Theta_{5555}  + 280 H_{5}^2 \Theta_{55} \Theta_{555}^2 \Theta_{5555} - 20 H H_{55} \Theta_{55} \Theta_{555}^2 \Theta_{5555} \\ 
     & +     80 H H_{5} \Theta_{555}^3 \Theta_{5555} - 35 H_{5}^2 \Theta_{55}^2 \Theta_{5555}^2  -     36 H H_{55} \Theta_{55}^2 \Theta_{5555}^2 \\
     & - 140 H H_{5} \Theta_{55} \Theta_{555} \Theta_{5555}^2  - 20 H^2 \Theta_{555}^2 \Theta_{5555}^2 + 30 H^2 \Theta_{55} \Theta_{5555}^3 \\ 
     & + 10 H H_{555} \Theta_{55}^3 \Theta_{55555} + 70 H H_{55} \Theta_{55}^2 \Theta_{555} \Theta_{55555}  + 80 H H_{5} \Theta_{55} \Theta_{555}^2 \Theta_{55555} \\
     & - 14 H H_{5} \Theta_{55}^2 \Theta_{5555} \Theta_{55555} - 10 H^2 \Theta_{55} \Theta_{555} \Theta_{5555} \Theta_{55555} - 5 H^2 \Theta_{55}^2 \Theta_{55555}^2 \\ 
     & +     3 H^2 \Theta_{55}^2 \Theta_{5555} \Theta_{555555})\\
b_{52}&=-\frac{1}{3 \Theta_{5555}^2} (-H_{555} \Theta_{55}^2 - 7 H_{55} \Theta_{55} \Theta_{555} - 8 H_{5} \Theta_{555}^2 + 5 H_{5} \Theta_{55} \Theta_{5555} \\
     &+ 4 H \Theta_{555} \Theta_{5555} + H \Theta_{55} \Theta_{55555}).
\end{aligned}
$$
Even in this restricted situation the computed Cartan's matrix coefficients $A_i$, $i=1,2,3 ,4, 5$, are too complicated to list them all here. However, surprisingly we found that 
$$ \begin{aligned}
A_5&=\frac{1}{100 \Theta5555^4} (175 \Theta_{55555}^4 - 280 \Theta_{5555} \Theta_{55555}^2 \Theta_{555555} + 49 \Theta_{5555}^2 \Theta_{555555}^2 \\
&+ 70 \Theta_{5555}^2 \Theta_{55555} \Theta_{5555555} - 10 \Theta_{5555}^3 \Theta_{55555555}).
\end{aligned}
$$
For $G_2$ flatness of the twistor distribution we need that all the $A_i$, $i=1,2,3,4,5$ must vanish. The form of the computed $A_5$ shows that Noth's equation
\be\begin{aligned}
  175 \Theta_{55555}^4 - 280 \Theta_{5555} &\Theta_{55555}^2 \Theta_{555555} + 49 \Theta_{5555}^2 \Theta_{555555}^2 +\\
&70 \Theta_{5555}^2 \Theta_{55555} \Theta_{5555555} - 10 \Theta_{5555}^3 \Theta_{55555555}=0,\end{aligned}\label{no5}\ee
for $\Theta_5$ is a necessary condition. Whether this condition is sufficient it is an open question.

Interestingly, there are three well-known solutions of the Monge-Ampere (\ref{ma1}):
\begin{enumerate}
\item $\Theta= \phi (C_1 x + C_2 y ) + C_3 x+ C_4 y +C_5 $
\item $\Theta= (C_1 x+ C_2 y)\phi \left( \frac{y}{x} \right) + C_3 x +C_4 y + C_5$ 
\item $\Theta= (C_1 x + C_2 y +C_3) \phi \left( \frac{C_4 x+ C_5 y + C_6}{C_1 x + C_2 y + C_3}\right) + C_7 x + C_8 y + C_9$
\end{enumerate}
Rewriting these solutions in terms of the variables $x_3,x_4$ and $x_5$ using $x=x_5$, $y=x_4-x_3 x_5$, we impose the Noth's equation (\ref{no5}) on such $\Theta$s, which in turn impose conditions on the free function $\phi$ appearing in (1), (2), (3). The results are as follows:
\begin{itemize}
\item[ad (1)] In this case all the $A_i$s, $i=1,2,3,4,5$, vanish provided that the function $\phi_t$ satisfies the Noth's ODE for the variable $t=C_1 x_5+C_2(x_4-x_3 x_5)$. Thus in this case Noth's ODE is \emph{sufficient} to guarantee the $G_2$ flatness of the twistor distribution. However, by a simple change of coordinates, one can convince himself that the heavenly metric corresponding to $\Theta=\phi (C_1 x + C_2 y ) + C_3 x+ C_4 y +C_5$ is the same as in Theorem \ref{theo31}.
\item[ad (2)-(3)] In these cases the equations $A_i=0$, $i=1,2,3,4,5$, cannot be satisfied by function $\phi$, which obey the generiticity condition $\Theta_{5555}\neq 0$. 
\end{itemize}

We actually have computed  $A_i$, $i=1,2,...,5$ for the Pleba\'nski metric satisfying the heavenly equations with heavenly function $\Theta$ depending on all the variables $(x,y,z,w)$. This was done with the help of Mathematica symbolic calculation program. It follows that in that, fully general case, the Noth equation (\ref{no5}) is \emph{still necessary} for the heavenly function $\theta$ to correspond to $G_2$ flat twistor distribution. It would be interesting to know if there are heavenly functions satisfying both (\ref{secpleb}) and (\ref{no5}) and for which the corresponding heavenly metric (\ref{ple}) is not equivalent to the one from Theorem \ref{theo31}. 

\section{Acknowledgements}
We take this oportunity to remember the work of Sotirios Bonanos. He made the exterior differential package for Mathematica, which was used by us for many years. Many times Sotirios updated the package responding to our requests. Without his work many papers of one of the authors (PN) would never been written. This referes also to the present paper.

\end{document}